\documentclass[reqno, 10pt]{amsart}

\usepackage[latin5]{inputenc}

\newtheorem{thm}{Theorem}

\theoremstyle{definition}

\theoremstyle{remark}

\numberwithin{equation}{section}
\newtheorem*{thm A}{Theorem A}

\begin{document}
\thispagestyle{empty}\pagestyle{myheadings}%
\markboth{\small Generating functions}{\small V.Gupta et al. }
\begin{center}
\textbf{{\Large Generating functions for $q$-Bernstein,\\ $q$-Meyer-K\"{o}nig-Zeller and $q$-Beta basis}} \vskip0.3in
\textbf{Vijay Gupta}\\
{School of Applied Science, Netaji Subhas Institute of Technology\\ Sector 3 Dwarka
 110078 New Delhi, India}\\{vijaygupta2001@hotmail.com}
\vskip0.3in
\textbf{Taekyun Kim, Jongsung Choi and Young-Hee Kim}\\
{Division of General Education-Mathematics, Kwangwoon University\\
Seoul 139-701, S. Korea,}\\
 {tkkim@kw.ac.kr, jeschoi@kw.ac.kr, yhkim@kw.ac.kr}\vskip0.3in

\end{center}
 \vskip0.3in \textbf{Abstract.}   The present paper deals with the $q$-analogues of Bernstein,
 Meyer-K\"{o}nig-Zeller and Beta operators. Here we estimate the generating functions for $q$-Bernstein, $q$-Meyer-K\"{o}nig-Zeller and $q$-Beta basis functions.
   \\ \vskip0.1in
\textbf{Keywords:} $q-$integers, $q-$binomial coefficient, $q$-exponential, $q$-Bernstein basis, $q$-Meyer-K\"{o}nig-Zeller basis and $q$-Beta basis function.\\

\textbf{Mathematical subject classification:} 41A25, 41A35.\\

\section{\bf Introduction}
For each non-negative integer $k$, the
$q-$integer $[k]_q$ and the $q-$factorial $[k]_q!$ are respectively
defined by $$ [k]_q=\left\{\begin{array}{ll} (1-q^k)\big / (1-q),\
\ &q\neq 1\\k,&q=1\end{array},\right.$$ and $$
[k]_q!=\left\{\begin{array}{ll} [k]_q\, [k-1]_q\cdots [1]_q,\ \ &k\ge
1\\1,&k=0\end{array}.\right.\  $$ For the integers $n,\ k $
satisfying $ n\ge k\ge 0$, the $q-$binomial coefficients are
defined by
$$\left[\begin{array}{c}n\\k\end{array}\right]_q=\frac{[n]_q!}{[k]_q![n-k]_q!}\
$$
(see e.g. \cite{TE}).
We consider the $q$-exponential function in the following form:
 \begin{eqnarray*}
 \lim_{n\rightarrow \infty}\frac{1}{(1-x)_q^n} &=&   \lim_{n\rightarrow \infty}\sum_{k=0}^\infty  \left[ {\begin{array}{*{20}{c}}
  {n+k-1}  \\
  k  \\
 \end{array}} \right]_q x^k \\
  &=& \lim_{n\rightarrow \infty}\sum_{k=0}^\infty \frac{(1-q^{n+k-1})....(1-q^n)}{(1-q)(1-q^2)...(1-q^k)}x^k\\
  &=& \sum_{k=0}^\infty \frac{x^k}{(1-q)(1-q^2)...(1-q^k)}=e_q(x).
   \end{eqnarray*}
   Another form of $q$-exponential function is given as follows:
      \begin{eqnarray*}
\lim_{n\rightarrow \infty}(1+x)_q^n=\sum_{k=0}^\infty \frac{q^{k(k-1)/2}x^k}{(1-q)(1-q^2)...(1-q^k)}=E_q(x).
\end{eqnarray*}
It is easily observed that $e_q(x)E_q(-x)=e_q(-x)E_q(x)=1$.\\

Based on the $q$-integers Phillips \cite{Phillips} introduced the $q$ analogue of the well known Bernstein polynomials. For $f\in C[0,1]$ and $0<q<1$, the $q$-Bernstein polynomials are defined as
\begin{equation}\label{1}
{\mathcal{B}_{n,q}}\left( {f,x} \right) = \sum\limits_{k = 0}^n  {b_{k,n}^q(x)} f\left( \frac{[k]_q}{[n]_q} \right),
\end{equation}
where the $q$-Bernstein basis function is given by
$$b_{k,n}^q(x) = \left[ {\begin{array}{*{20}{c}}
   {n}  \\
   k  \\
\end{array}} \right]_q x^k (1 - x)_q^{n - k}, x\in [0,1]$$
and $\left(a-b\right)_{q}^{n}=\prod\limits_{s=0}^{n-1}(a-q^{s}b),\quad a,b\in\mathbb{R}.$

Also Trif \cite{Trif} proposed the $q$ analogue of well known Meyer-K\"{o}nig-Zeller operators. For $f\in C[0,1]$ and $0<q<1$, the $q$-Meyer-K\"{o}nig-Zeller operators are defined as
\begin{equation}\label{1}
{\mathcal{M}_{n,q}}\left( {f,x} \right) = \sum\limits_{k = 0}^\infty  m_{k,n}^q(x) f\left( \frac{[k]_q}{[n]_q} \right),
\end{equation}
where the $q$-MKZ basis function is given by
$$m_{k,n}^q(x) = \left[ {\begin{array}{*{20}{c}}
   {n+k+1}  \\
   k  \\
\end{array}} \right]_q x^k (1 - x)_q^{n}, x\in [0,1].$$

For $f\in C[0,\infty)$ and $0<q<1$, the $q$-Beta operators are defined as
       \begin{equation}\label{1}
       {\mathcal{V}_{n}}\left( {f,x} \right) = \frac{1}{[n]_q}\sum\limits_{k = 0}^\infty  v_{k,n}^q(x) f\left( \frac{[k]_q}{q^{k-1}[n]_q} \right),
 \end{equation}
  where the $q$-Beta basis function is given by    $$v_{k,n}^q(x) = \frac{q^{k(k-1)/2}}{B_q(k+1,n)}\frac{x^k}{(1+x)_q^{n+k+1}}, x\in [0,\infty )$$
       and $B_q(m,n) $ is $q$-Beta function.\\

In the present paper we establish the generating functions for $q$-Bernstein, $q$-Meyer-K\"{o}nig-Zeller and $q$-Beta basis functions.

\section{\bf Generating Function for $q$-Bernstein basis}
\begin{thm}\label{thm:1}
$b_{k,n}^q(x) $ is the coefficient of $\frac{t^n}{[n]_q!}$ in the expansion of
$$\frac{x^kt^k}{[k]_q!}e_q((1-q)(1-x)_qt).$$
\end{thm}
\begin{proof}
First consider
\begin{eqnarray*}
\frac{x^kt^k}{[k]_q!}e_q((1-q)(1-x)_qt)&=&\frac{x^kt^k}{[k]_q!}\sum_{n=0}^\infty \frac{(1-x)_q^nt^n}{[n]_q!}\\
&=&\frac{1}{[k]_q!}\sum_{n=0}^\infty \frac{x^k(1-x)_q^nt^{n+k}}{[n]_q!}\\
&=&\sum_{n=0}^\infty \frac{[n+1]_q[n+2]_q.....[n+k]_qx^k(1-x)_q^nt^{n+k}}{[n+k]_q![k]_q!}\\
&=&\sum_{n=0}^\infty \left[\begin {array}{c}
        n+k \\
        k
      \end{array}
\right]_q\frac{x^k(1-x)_q^nt^{n+k}}{[n+k]_q!}\\
 &=&\sum_{n=k}^\infty \left[\begin {array}{c}
 n \\
  k
   \end{array}
     \right]_q  \frac{x^k(1-x)_q^{n-k}t^n}{[n]_q!}=\sum_{n=0}^\infty b_{k,n}^q(x)\frac{t^n}{[n]_q!}.
\end{eqnarray*}
This completes the proof of generating function for $b_{k,n}^q(x).$
\end{proof}
\section{\bf Generating Function for $q$-MKZ basis}
\begin{thm}\label{thm:1}    $m_{k,n}^q(x) $ is the coefficient of $t^k$ in the expansion of $\frac{(1-x)_q^n}{(1-xt)_q^{n+2}}.$
 \end{thm}
    \begin{proof}
   It is easily seen that
   \begin{eqnarray*}
 \frac{(1-x)_q^n}{(1-xt)_q^{n+2}}=  \sum_{k=0}^\infty  \left[\begin {array}{c}
  n+k+1 \\
   k
   \end{array}
     \right]_q(1-x)_q^n x^kt^k= \sum_{k=0}^\infty m_{k,n}^q(x) t^k .
 \end{eqnarray*}
 This completes the proof.
      \end{proof}
      \section{\bf Generating Function for $q$-Beta basis}
 \begin{thm}        It is observed by us that $v_{k,n}^q(x)$ is the coefficient of
         $\frac{t^k}{[n+k]_q!}$ in the expansion of $\frac{1}{(1+x)_q^{n+1}}E_q\left(\frac{(1-q)xt}{(1+q^{n+1}x)_q}\right).$
 \end{thm}
  \begin{proof}
  First using the definition of $q$-exponential $E_q(x),$ we have
\begin{align*}
\frac{1}{(1+x)_q^{n+1}}E_q\left(\frac{(1-q)xt}{(1+q^{n+1}x)_q}\right)&=\frac{1}{(1+x)_q^{n+1}}\sum_{k=0}^\infty q^{k(k-1)/2}\frac{x^k}{(1+q^{n+1}x)_q^k}\frac{t^k}{[k]_q!}\\
&=\sum_{k=0}^\infty q^{k(k-1)/2}\frac{x^k}{(1+x)_q^{n+k+1}}\frac{t^k}{[k]_q!}\\
&=\sum_{k=0}^\infty q^{k(k-1)/2}\frac{x^kt^k}{(1+x)_q^{n+k+1}}\frac{[k+1]_q[k+2]_q...[n+k]_q}{[n+k]_q!}\\
&=\sum_{k=0}^\infty q^{k(k-1)/2}\frac{x^k}{(1+x)_q^{n+k+1}}{\left[ {\begin{array}{*{20}{c}}
   {n + k }  \\
   n  \\
\end{array}} \right]_q}\frac{[n]_q t^k}{[n+k]_q!}\\
&=\sum_{k=0}^\infty \frac{1}{B_q(k+1,n)} q^{k(k-1)/2}\frac{x^k}{(1+x)_q^{n+k+1}}\frac{ t^k}{[n+k]_q!}\\
&=\sum_{k=0}^\infty v_{k,n}^q(x)\frac{ t^k}{[n+k]_q!}.
\end{align*}
This completes the proof of generating function.\\
 \end{proof}
 {\bf Acknowledgement:} The present work was done when the first author visited Division of General Education-Mathematics, Kwangwoon University, Seoul, S. Korea during June 2010.


\begin{thebibliography}{99}
\bibitem{Phillips}   G. M. Phillips, On generalized Bernstein polynomials. In: Griffiths, D. F., Watson, G.A.(eds): Numerical analysis. Singapore:
World Scientific 1996, pp.\ 263-269. \bibitem{Trif} T. Trif, Meyer-K\"{o}nig and Zeller operators based on the q-integers, Rev. Anal. Numer. Theor. Approx. 29 (2000), No. 2, 221-229.
\bibitem{TE} T.  Ernst, The history of $q-$calculus and a new method, U.U.D.M. Report 2000, 16, Uppsala, Departament of Mathematics, Uppsala University (2000).
\end{thebibliography}
\end{document}